%% file: pdimPaperArxiv.tex
\documentclass[a4paper]{article}
\usepackage{amsthm,amssymb,amsmath,enumerate,graphicx,epsf,mathtools}

\newcommand{\COLORON}{1}
\newcommand{\NOTESON}{1}
\newcommand{\Debug}{0}
\input{defs}

\newcommand{\la}{\ensuremath{\lambda}}
\newcommand{\cme}{\ensuremath{\prescript{}{e}\cm}}
\newcommand{\coco}{\ensuremath{\nu}}

\newcommand{\epdim}[1]{\ensuremath{\prescript{}{e}{{pdim#1}}}}

\newcommand*{\zp}[1][Z]{\mathbb{\uppercase{#1}}_{+}}
\newcommand{\pdim}{\ensuremath{pdim}}
\newcommand{\ndim}{\ensuremath{\coco dim}}

\renewcommand{\labtequc}[2]{ \begin{equation} \label{#1} 	\text{   #2 }\end{equation} }

    \newenvironment{itemize*}%
      {\begin{itemize}%
        \setlength{\itemsep}{0pt}%
        \setlength{\parskip}{0pt}}%
      {\end{itemize}}

    \newenvironment{enumerate*}%
      {\begin{enumerate}%
        \setlength{\itemsep}{8pt}%
        \setlength{\parskip}{-8pt}}%
      {\end{enumerate}\vspace*{-8pt}
}

\usepackage{enumerate}

\title{A Notion of Dimension based on Probability on Groups}

\begin{document}

\author{Agelos Georgakopoulos\thanks{Supported by  EPSRC grants EP/V048821/1, EP/V009044/1, and EP/Y004302/1.}}
\affil{  {Mathematics Institute}\\
 {University of Warwick}\\
  {CV4 7AL, UK}}

\date{}

\maketitle

\vspace*{-1.1cm}

\begin{abstract}
We introduce notions of dimension of an infinite group, or more generally, a metric space, defined using percolation. Roughly speaking, the percolation dimension $pdim(G)$ of a group $G$ is the fastest rate of decay of a symmetric probability measure $\mu$ on $G$, such that Bernoulli percolation on $G$ with connection probabilities proportional to $\mu$ behaves like a Poisson branching process with parameter 1 in a sense made precise below. We show that $pdim(G)$ has several natural properties: it is monotone decreasing with respect to subgroups and quotients, and coincides with the growth rate exponent for several classes of groups.
\end{abstract}

\section{Background and Motivation}
We introduce new invariants of infinite groups that behave similarly to the growth rate but are less susceptible to the effect of dead-ends. We will define a notion of dimension $\pdim(G)$ of a finitely generated group \g using (long-range, Bernoulli) percolation. It will turn out to have several natural properties. In particular, for groups of polynomial growth it follows from parallel work of Spanos \& Tointon \cite{SpaToiSpr} that $\pdim(G)$ coincides with the growth exponent. Importantly, unlike most parameters defined via percolation, $\pdim(G)$ does not depend on the choice of a generating set.

\smallskip

\comment{
\begin{figure}[t]
  \centering
    \begin{tikzpicture}[font=\small, scale=1.5]
      \input{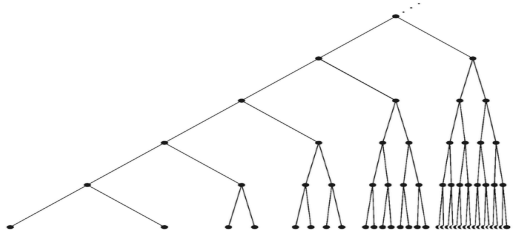}
    \end{tikzpicture}
  \caption{\label{fig:canopy} The \defi{canopy tree} $T$ consists of an 1-way infinite path $x_0 x_1 \ldots$ and a binary tree of depth $i$ attached to each $x_i$.}
\end{figure}
}

The growth rate is one of the most important invariants of infinite groups due to its interplay with other properties of the group, as witnessed by Gromov's polynomial growth theorem or the fact that non-amenable groups have exponential growth. Nevertheless, there are cases where the growth of a group can be somewhat deceiving. Let me start explaining this with an example that is not a group, but much easier to explain. The graph in \fig{canopy} is the so-called \defi{canopy tree} $T$, frequently used as an example of a local limit of an infinite sequence of finite graphs\footnote{More specifically, $T$ is the Benjamini--Schramm limit of \seq{T}, where $T_n$ stands for the binary tree of depth $n$ \cite{BeSchrRec,NacPla}; but this is not important here.}. It is easy to see that $T$ has exponential growth. However, in certain respects it behaves like an `1-dimensional' graph: for example, it has $p_c(T)=1$, i.e.\ it has a trivial phase transition for Bernoulli percolation. Indeed, $T$ can be obtained by attaching finite trees along an 1-way infinite path $P$, and it will behave much more like $P$ than the infinite binary tree with respect to most statistical mechanics models despite its large growth rate. One explanation to this is that the vast majority of directions that a walker on $T$ can explore lead into \defi{dead-ends}, i.e.\ they cannot be extended to increase the distance from a fixed vertex. 

   \begin{figure}[htbp]
   \centering
   \noindent
   \includegraphics[width=.55\linewidth, height = .25\linewidth]{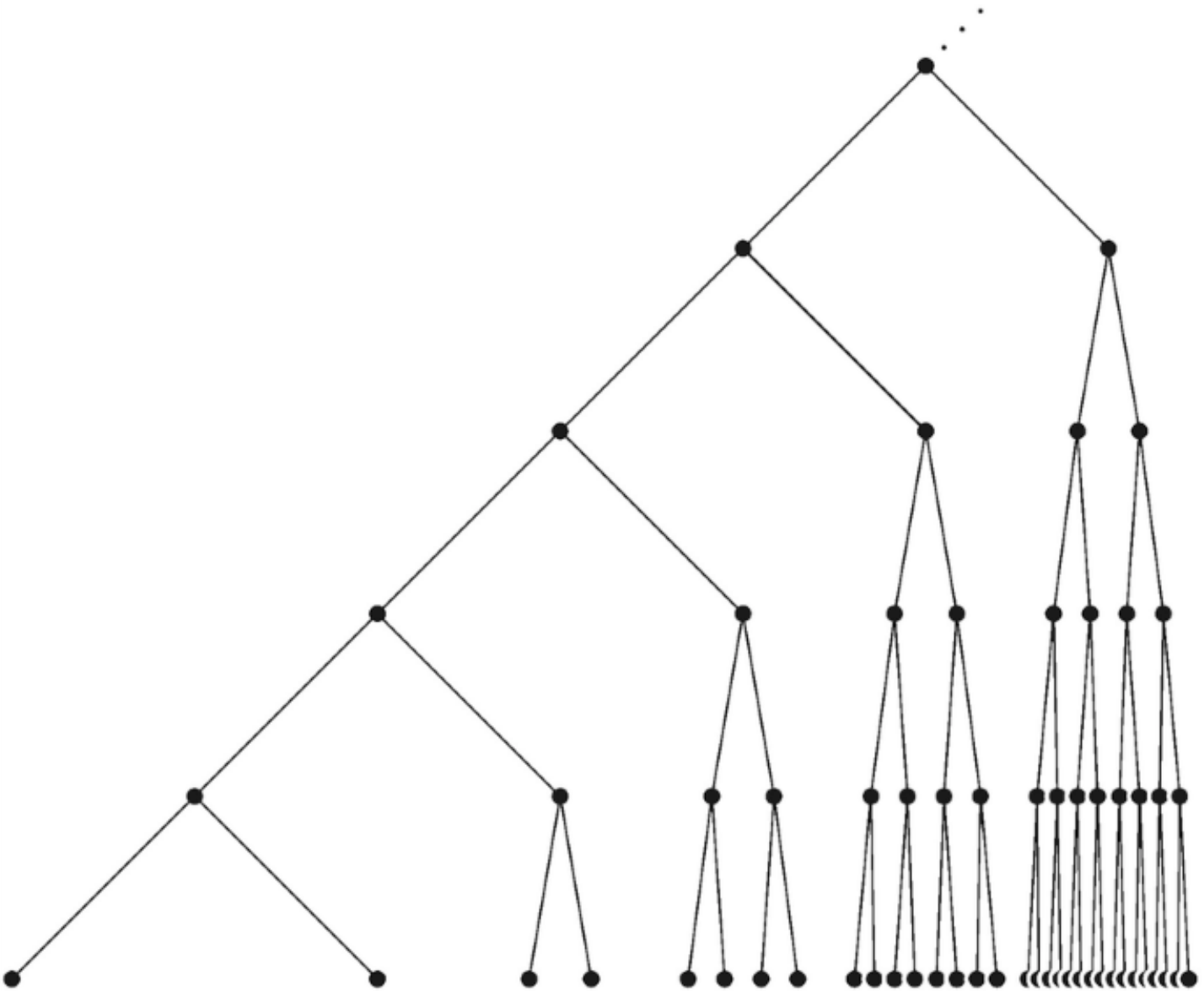}
   \caption{\small The \defi{canopy tree} $T$ consists of an 1-way infinite path $x_0 x_1 \ldots$ and a binary tree of depth $i$ attached to each $x_i$.
   }
   \label{canopy}
   \end{figure}
   

Such phenomena arise in \Cg s of groups, although in much more subtle ways. Notable examples include Wilson's groups of non-uniform exponential growth \cite{WilUni}, and  the groups of oscillating intermediate growth of Kassabov \& Pak \cite{KasPakOsc}. 

The notion of dimension we will  introduce below assigns dimension 1 to the above example $T$. It has several desirable properties such as being monotone decreasing with respect to subgroups and quotients of groups. 

\medskip

Statistical mechanics usually studies physical phenomena, but there is a recent trend of studying its  stochastic processes on (\Cg s of) arbitrary groups, and making a connection between the algebraic or geometric properties of the group and the behaviour of the process. The most studied examples are random walk and (Bernoulli) percolation. The pioneering result in this direction is Kesten's theorem that the \rw\ return probability decays exponentially \iff\ the group is non-amenable \cite{KesFul}. An analogous result of Benjamini \& Schramm for percolation 
states that the group $G$ is non-amenable \iff\ there is a $G$-invariant Bernoulli percolation model with a non-uniqueness phase \cite{BeSchrPer}.  Such results initiated the area of probability on groups, some of the recent advances of which involve deep machinery from both worlds of probability and group theory, see  \cite{DCGRSY,EasHutLoc,HutToinNon,PanSevGap,SpaToiSpr,TesToiBal} for some examples.

Typically, such results use our understanding of groups in order to deduce properties of the 
stochastic process. As the wealth and depth of such results grows, the question of whether we can go in the opposite direction becomes pressing: can percolation theory be used as a tool to obtain deterministic results about groups? I hope that the notions defined here, being numerical invariants of a groups independent of a choice of a \Cg,  may make a step in this direction. 

Additional motivation for the notions defined here comes from Gromov's  \cite{GroAsyInv}  \defi{asymptotic dimension} $asdim(G)$ of a group $G$. Like $asdim(G)$, the dimension we will define is designed so as to capture the large scale geometry of a group, it generalises to arbitrary metric spaces, and has desirable properties like monotonicity \wrt\ subgroups. 

The idea of `embedding' a stochastic process into a space in order to study its coarse metric properties is not new, notions with a similar flavour include Ball's \cite{Ball} Markov type and  Gaboriau's  \cite{GabSewCos} cost of a group.

\bigskip
This note is structured as follows. We start with the formal definition of $\pdim(G)$ in \Sr{sec def}, and provide some basic properties in \Sr{sec prop}. The highlight of this section is \Cr{corol}, stating that $\pdim(G)$ coincides with the exponent of the growth rate when \g has polynomial growth. In \Sr{sec nu} we offer a variant $\ndim(G)$ of $\pdim(G)$ defined using self-avoiding walks instead of percolation, and show that $\ndim(G) = \pdim(G)$ when \g has polynomial growth (\Cr{corol nu}). In \Sr{sec interm} we adapt our notions to groups of super-polynomial growth. The main result here is that every non-amenable group achieves the maximum possible value of $\pdim(G)$ (\Tr{thm nona}). \Sr{sec PG} introduces a notion of percolative group that is implicit in the aforementioned results but might be of independent interest to percolation theorists. \Sr{sec BG} extends some of our definitions to metric spaces. We conclude with some open problems in \Sr{probs}.

\section{Definition of the dimension of a group} \label{sec def}

Given a group \G, we call a probability measure $\mu: G \to [0,1]$ \defi{symmetric}, if $\mu(g)=\mu(g^{-1})$ \fe\ $g\in G$. Let $\cm(G)$ denote the set of all symmetric probability measures on \G. For simplicity \g will be finitely generated, but our definitions can be adapted to topological groups endowed with a metric.

One can use $\mu$ to define an \defi{edge-weighted \Cg} of \G, in which the generating set is the support of $\mu$ and every edge $gh$ is given the weight $\mu(g^{-1}h)$. With this notion in mind, it is rather straightforward to adapt any model of statistical mechanics defined on graphs into one where the `strengths of interactions' are given by $\mu$. Thus any $\mu \in  \cm(G)$ can be used to define a plethora of models. Examples include \rw, various percolation models, self-avoiding walks, and the Ising model.


Any $\mu \in  \cm(G)$  naturally defines a percolation process on $G$ as follows. Given $\la \in \zp[r]$, we define a random (multi-)graph $G^\mu(\la)$ with vertex set $G$, by letting the number of (parallel) edges between two elements $g,h\in G$ be an independent Poisson random variable with mean $\la \mu(g^{-1}h)$. Note that  $G^\mu(\la)$  is a $G$-invariant percolation model, i.e.\ the natural action of $G$ on $G^\mu(\la)$ defined by multiplication from the left preserves the probability distribution of $G^\mu(\la)$. This setup is a generalisation of Long Range Percolation, which has been studied in statistical mechanics since the 80's, see e.g.\ \cite{AizKeNew}. 
Similarly to the standard percolation threshold $p_c$, we  define 
\labtequc{def lac}{$\la_c(\mu):= \sup \{ \la \mid \Pr(G^\mu(\la) \text{ has an infinite component}) = 0 \}.$}

Note that $\la_c(\mu)\geq 1$ since an exploration of the component of the identity in $G^\mu(\la)$ is dominated by a Poisson branching process with parameter \la.  Answering a question from \cite{IGperc}, Xiang \& Zou proved that every countable group admits a symmetric measure $\mu$ such that $\la_c(\mu)$ is finite \cite{XiaZouCou}, thereby generalising the non-triviality of $p_c$ \cite{DCGRSY} to locally finite groups.


Apart from percolation, one can use the same $\mu$ to define long-range versions of other models, e.g.\ the connective constant. We will elaborate on this in \Sr{sec nu}.

This is more than just an economical way for constructing many models at once. Firstly, having a common defining measure $\mu$ allows for some direct comparisons between various models (see \eqref{lacoco} below for an example). Secondly, our probability measures  $\cm(G)$ offer in some sense a `normalised' framework: it is more informative to compare the two percolation models on a common group $G$ arising from $\mu_1,\mu_2\in \cm(G)$ than the percolation models on two \Cg s of \g with different degrees.

\comment{
Most of the above discussion extends to setups beyond groups: we could replace $G$ by any g
}

\smallskip
We now add geometry to the picture by recalling that any choice of a \Cg\ $Cay(G,S)$ of a finitely generated group \g \wrt\ a generating set $S$ turns \g into a metric space, and the `large-scale' properties of the metric are independent of the choice of $S$. This will allow us to associate a \defi{rate of decay} $s\in \R_+$ to any $\mu \in  \cm(G)$, independent of the choice of $S$. {Roughly speaking, we then define the dimension of \g as the supremal $s$ \st\ \ti\ a sequence of $\mu_i \in  \cm(G)$, each of rate of decay $s$, and exhibiting `mean-field behaviour'.} When defining  \defi{percolation dimension} $pdim(G)$, mean-field behaviour means that $\lim_{i\to \infty} \la_c(\mu_i) = 1$, where $\la_c$ is defined by \eqref{def lac}. In other words, it means that the model approximates the Poisson Branching Process in the sense that at criticality, the average degree of the root vertex converges to 1 as $i\to \infty$. 

\medskip

We now state the formal definition of $pdim(G)$. Let $X = Cay(G,S)$ be a \Cg\ of \g \wrt\ a finite generating set $S$ (hence \g is finitely generated). Let $|g|:= d_X(g,e)$ denote the distance from $g$ to the identity $e$ \wrt\ the graph metric of $X$. For any $s,b\in \R_+$, we define the set of symmetric probability measures $\cm_s^b(G) \subset \cm(G)$ with decay rate $s$ (and constant $b$) as follows:
\labtequc{msb}{$\cm_s^b(G):= \{ \mu \in \cm(G) \mid \mu(g) < b |g|^{-s}\}.$}

\begin{definition}
The \defi{percolation dimension} $pdim(G)$ of \g is defined by
\labtequc{def pdim}{$pdim(G):= \sup \left\{ s \mid  \exists b \in \R_+ \text{ such that } \liminf_{\mu\in \cm_s^b(G)} \la_c(\mu) =1 \right\}.$}
\end{definition}

Note that although $|g|$, and hence $\cm_s^b(G)$, depends on the choice of $S$, the value of $pdim(G)$ does not: different choices of $S$ may impose a different choice of $b$ in the above definition, but they cannot affect the existence of such a constant. This is similar to the standard calculation showing that the growth rate of $G$ is independent of the choice of $S$.

\medskip
Rather than defining one concept, \eqref{def pdim} can be thought of as a recipe for defining many such concepts: we can replace percolation by many other models; for example we will define the \defi{connective dimension} in \Sr{sec nu}. In \eqref{msb} we made the decay rate polynomial in $|g|$, but we can change it to stretched-exponential; see \Sr{sec interm}. Finally, we can replace $X = Cay(G,S)$ by an arbitrary graph or metric space $X$ and adapt the definition of $\mu$ by dropping its symmetry; see \Sr{sec BG}. I think of this versatility as an advantage: it is a tool that can be taylored to the desired application.

\section{Properties of $pdim(G)$} \label{sec prop}

It is far from clear from the definitions that $pdim(G)$ is a non-trivial quantity, let alone a well-behaved one. The following establishes this:
\begin{proposition} \label{prop d}
\Fe\ $d\geq 1$, we have $pdim(\Z^d)= d$. 
\end{proposition}
More generally, we will see below that for every \g of polynomial growth, $pdim(G)$ equals the growth exponent. In particular, $pdim(G)$ turns out to be an integer for such groups.

To prove these claims we will use the following proposition, which states that the exponent of the (polynomial) growth rate of \g upper bounds $pdim(G)$. Given a \Cg\ $X$, we let $B_X(n)$ denote the ball of radius $n$ around the identity in $X$.
\begin{proposition} \label{prop UB}
Let $X$ be a \Cg\ of a group \g \st\ $|B_X(n)|< cn^k$ for some $c,k\in \R$ and every $n\in \N$. Then $pdim(G)\leq k$.
\end{proposition}
\begin{proof}
Suppose, to the contrary, that $pdim(G)\geq k+\epsilon$ for some $\epsilon>0$. Pick some $M\in \N$ (which will depend on $\epsilon,c,k$ in a way that we make precise below), and let $A_0:=B_o(M)$, and for every $i\in N_{>0}$,  define the ``annulus'' $A_i:= B_o(M 2^i) \sm B_o(M 2^{i-1})$.

Let $\mu\in \cm_{k+\epsilon}^b(G)$ for some $b\in \R$. Then \fe\ $i>0$, we have %
\labtequ{ineq mu}{$\mu(A_i) \leq c  (M 2^i)^k b (M 2^{i-1})^{-(k+\epsilon)}$}
 because $A_i$ has at most $|B_o(M 2^i)| \leq c  (M 2^i)^k$ elements,  each of which elements $g$ has $\mu(g)\leq b (M 2^{i-1})^{-(k+\epsilon)}$ by the definitions. Rearranging \eqref{ineq mu} by cancelling out the factor $(M2^i)^k$ we deduce $\mu(A_i) \leq b c 2^{k+\epsilon} M^{-\epsilon} 2^{-i \epsilon}$. Note that the only term depending on $i$ here is $2^{-i \epsilon}$, which decays exponentially in $i$, and in particular it is summable. Therefore, 

$$\mu(A_0) = 1 - \sum_{i=1}^\infty \mu(A_i) \geq 1- b c 2^{k+\epsilon} M^{-\epsilon} \sum_{i=1}^\infty 2^{-i \epsilon}.$$

Easily, we can choose $M=M(b,c,\epsilon,k)$ large enough that $\mu(A_0)>1/2$.

Let $\mu_j, j\in \N$ be a sequence of elements of $\cm_{k+\epsilon}^b(G)$ witnessing the fact that $pdim(G)\geq k+\epsilon$, i.e.\ $\lim_{j\to \infty} \la_c(\mu_j) =1$. Since the above calculations did not depend on the choice of $\mu$, we deduce that $\mu_j(A_0)>1/2$ \fe\ $j$.

Since $A_0$ is fixed and finite, we deduce that 
\ti\ $\delta>0$ \st\ \fe\ $j$ \ti\ $g_j$ with $\mu_j(g_j)>\delta$. 
It follows that we expect  $\delta' \approx \delta$ parallel edges between the identity $o$ and $g_j$, and so the expected  simple degree of $o$ is at most $\la- \delta'$, which is less than 1 when $\la \to 1$. This contradicts our assumption $\la_c(\mu_k) \to 1$ by comparison with the  Poisson branching process with rate $\la- \delta'$ and the argument we appealed to after \eqref{def lac}.
\end{proof}

{\bf Remark:} In the last proof we only used the assumption $|B_o(n)|< cn^k$ on a sequence $n_i$ of values of $n$ forming a geometric progression.  

\smallskip



\begin{proof}[Proof of \Prr{prop d}]
The lower bound $\pdim(\Z^d)\geq d$ can be deduced from a result of  \cite{PenSpr}, re-proved in \cite{BoJaRiSpr}, or from results of \cite{HeHoSaMe}. 
The upper bound $\pdim(\Z^d)\leq d$ is provided by \Prr{prop UB}.
\end{proof}

In parallel work, Spanos \& Tointon \cite[Theorem 1.1]{SpaToiSpr} prove a deep result extending  the aforementioned lower bounds to all groups of polynomial growth. Combining this result with \Prr{prop LB} below, and again using \Prr{prop UB} for the upper bound, we immediately obtain the following generalisation of \Prr{prop d}:

\begin{corollary} \label{corol}
For every group \g of polynomial growth $\Theta(n^d)$, we have\\ $\pdim(G)= d$.
\end{corollary}


\smallskip
Next, we remark that if $H$ is a subgroup of $G$, then $\pdim(H)\leq \pdim(G)$, just because $\cm_s^b(H) \subseteq \cm_s^{b'}(G)$ when choosing the \Cg s appropriately. Moreover, one can show that if $H$ is a quotient of $G$, then $\pdim(H)\leq \pdim(G)$, by lifting any measure $\mu_i\in \cm_s(H)$ to a measure in $\cm_s(G)$.

\section{Long-range connective constant, and connective dimension} \label{sec nu}

As mentioned above, we can replace percolation by many other models of statistical mechanics in the definition of $pdim$. \mymargin{e.g.\ for First Passage Percolation, the Ising model, counting lattice animals, etc.} In this section we elaborate on one such example, based on the connective constant.

Recall that a \defi{self-avoiding walk} (SAW) in a graph $X$ is a sequence $x_1 \ldots x_n$ of distinct vertices such that $x_i x_{i+1}$ is an edge \fe\ relevant $i$. The \defi{connective constant $\nu(X)$} of $X$ is the exponential growth rate  of the number $c_n=c_n(X)$ of SAWs of length $n$ starting at a fixed vertex; that is,  $\nu(X):= \lim c_n^{1/n}$.

The definition of $\nu(X)$ can be adapted to our long-range setup of a group \g endowed with a symmetric probability measure $\mu \in \cm(G)$ as follows.

To each sequence $S= (e=)g_0, g_1, \ldots, g_n$ of distinct elements of \G\ (which we think of as a `\defi{long-range SAW}'), we assign a weight $w_\mu(S):= \Pi_{1\leq i\leq n} \mu(g_i g_{i-1}^{-1})$, and define $ \sigma_n(\mu):= \sum_{S= (e=)g_0, g_1, \ldots g_n} w_\mu(S)$. We then define the \defi{connective constant}   $\coco(\mu)$ of $\mu$ by 
$\coco(\mu):=\lim_n \sigma_n(\mu)^{1/n}$.

If $X=Cay(G,S)$, and $\mu$ is the equidistribution on $X$, then 
\labtequc{numu}{$\nu(\mu)=\nu(X)/|S|$,} 
because $ \sigma_n(\mu)= c_n(X)/|S|^n$. 

Thus $\coco(\mu)$ generalises $\nu(X)$  in the same way that $\la_c(\mu)$ generalises $p_c(X)$. Grimmett \& Li have obtained generalisations of results about $\nu(X)$ to $\coco(\mu)$ \cite{GriLiWei}.  Moreover, similarly to the well-known inequality $p_c(X) \geq \frac1{\coco(X)}$ \cite{Lyonsbook} for a graph $X$, one can prove the inequality 
\labtequc{lacoco}{$\la_c(\mu) \geq \frac1{\coco(\mu)}$.}

In analogy to $pdim(G)$, we define the \defi{connective dimension} $\ndim(G)$ of \g  by
\labtequc{def ndim}{$\ndim(G):= \sup \left\{ s \mid \exists b \in \R_+ \text{ such that } \limsup_{\mu\in \cm_s^b(G)} \coco(\mu) =1 \right\}.$}
Again, we interpret the condition $\limsup_{\mu\in \cm_s^b(G)} \coco(\mu) =1$ as being asymptotically tree-like. Indeed, it follows from \eqref{numu} that $\nu(\mu)\leq \frac{|S|-1}{|S|}<1$, and if $F_n$ is a free group freely generated by a set $S_n$ of size $n$, then $\lim \frac{|S_n|-1}{|S_n|} =1$.


Similar argument to those in \Prr{prop UB} upper bound $\ndim(\Z^d)$, and more generaly $\ndim(G)$ of any \g of polynomial growth $\Theta(n^d)$  by $d$. 
Indeed, as in that proof, for any $k>d$ and any sequence $\mu_j \in \cm_k^b, j\in \N$, we find elements $g_j\in G$ with $\mu_j(g_j)>\delta>0$ \fe\ $j$. It follows that $\nu(\mu_j)\leq 1-\delta/2$, because whenever a long-range SAW traverses an edge of the form $s g_j$, it cannot traverse the edge $g_j s$ in the next step, and the latter edge has weight $\mu_j(g_j^{-1})= \mu_j(g_j) >\delta$.

On the other hand, \eqref{lacoco} yields  $\ndim(G)\geq  \pdim(G)$ \fe\ \G. Putting these facts together, and combining with \Cr{corol}, proves 
\begin{corollary} \label{corol nu}
For every group \g of polynomial growth $\Theta(n^d)$, we have $$\ndim(G)= \pdim(G)=d.$$
\end{corollary}
This raises the metaconjecture that our dimensions do not depend on the choice of the model.

\section{Groups of super-polynomial growth.} \label{sec interm}

Usually, having large growth makes it easier to percolate. Nevertheless, the problem of whether every finitely generated group has a \Cg\ $X$ with $p_c(X)<1$ remained open long after it was known for groups of polynomial growth, and it was settled in \cite{DCGRSY}. In analogy, we ask

\begin{problem} 
Does $\pdim(G)=\infty$ hold \fe\ group of super-polynomial growth?
\end{problem}  

We expect groups of intermediate growth to pose the greatest challenge here. 
For such groups, it is more natural however to consider  the \defi{exponential percolation dimension} $\epdim(G)$, which we now introduce. This is defined exactly like $pdim(G)$, except that we now change the polynomial decay condition $\mu(g) < b |g|^{-s}$ in the definition of $\cm_s^b(G)$ into the stretched exponential decay $\mu(g)= O(r^{|g|^{s}})$ for some $r<1$ and $s\in (0,1]$. In other words, we let $\cme_s^r(G):= \{ \mu \in \cm(G) \mid \mu(g) < r^{|g|^{s}}\}$, and define $\epdim(G)$ as in \eqref{def pdim}, except that we replace $\cm$ by $\cme$. 

Thus $s=1$ corresponds to exponential decay. In analogy to \Prr{prop UB}, we obtain, by repeating the same arguments, the following bound:
\begin{proposition} \label{prop UBE}
Let $X$ be a \Cg\ of a group \g \st\ \\ $|B_X(n)|<r^{|n|^{s}}$ for some $r,s\in \R$ and every $n\in \N$. Then $\epdim(G)\leq s$.
\end{proposition}
 
It might be that every \g of exponential growth has $\epdim(G)=1$. We will prove this when \g is non-amenable: 

\begin{theorem} \label{thm nona}
For every finitely generated non-amenable group $G$, we have  $\epdim(G)= 1$.
\end{theorem}

The main idea of the proof of this was contributed by Gabor Pete (private communication).

\begin{proof}
Let $\Sigma$ be a finite symmetric generating set of $G$, and let $G_1:=Cay(G,\Sigma)$ be the corresponding Cayley graph. 
It was proved by Thom \cite{ThoRem} that \fe\ $k\in \N$ \ti\ a symmetric generating set $S_k\subseteq \Sigma^k$ of \g \st\\ $\rho_{S_k} \leq 4k \ln(|\Sigma|)\rho(\Sigma)^k$, where $\rho_S:= \rho(Cay(G,S))$ is the spectral radius\footnote{The \defi{spectral radius} $\rho(X)$ of a graph $X$ can be defined as $\limsup p_n^{1/n}$, where $p_n$ denotes the probability that random walk will be back to its starting vertex at step $n$. We will not work with its  definition here; we will only use certain inequalities involving $\rho(X)$.} of the \Cg\ corresponding to a finite symmetric generating set $S$ of $G$.  Moreover, it is well-known that $|S|^{-1} \leq \rho_{S}^2$ holds for every such $S\subset G$ (by comparison
with a $2S$-regular tree). Putting these two inequalities together we obtain
\labtequ{Thom}{$|S_k|^{-1} \leq \rho_{S_k}^2 \leq c k^2 \rho_{\Sigma}^{2k}$.}
Recall that Kesten's Theorem \cite{KesFul} states that $\rho_\Sigma<1$ whenever $G$ is non-amenable. Let $r:= 1/\rho_\Sigma$.

Define $\mu_k\in \cm(G)$ to be the equidistribution on $S_k$. The two extremes of \eqref{Thom} imply that $|S_k| > r^{k}$ for large enough $k$. 
Hence $\mu_k(g) \leq r^{-|g|}$ \fe\ $g\in S_k$ since $|g|\leq k$ by the choice of $S_k$. In other words, $\mu_k \in \cme_r^{1}$ for large enough $k$.

We claim that $\lim_{k\to \infty} \la_c(\mu_k) = 1$, which means that $\epdim(G)\geq 1$ by the definitions, and hence $\epdim(G)= 1$ by \Prr{prop UBE} and the fact that no group has super-exponential growth. Equivalently, since $\mu_k$ is the equidistribution on $S_k$, our percolation model governed by $\mu_k$ is identically distributed with standard Bernoulli bond percolation on $G_k:= Cay(G,S_k)$, so our claim can be reformulated as $\lim_{k\to \infty} d(G_k) p_c(G_k) = 1$, where $d(G_k)$ is the degree of the vertices of $G_k$.

For this, we recall the edge Cheeger constant $\iota(X)$ of a graph $X$, defined as $\inf_{F \subset V(X), |F|< \infty} |\partial_E F| / |F|$, where $\partial_E F$ denotes the set of edges with exactly one end-vertex in $F$. It follows from Kesten's theorem \cite{KesFul} that $\iota(X)\geq 1 - \rho(X)$. Thus we have  $lim_k \iota(G_k) \geq 1$ by the above remark.

Finally, by a theorem of Benjamini \& Schramm  we have $p_c(X)< \frac1{d(X) \iota(X) }$ for every graph $X$ with $\iota(X)>0$, see \cite[formula (12.13)]{PeteBook}. Putting these facts together, we obtain $\lim_k d(G_k) p_c(G_k) \leq 1$, which proves our claim. 
\end{proof}

\subsection*{The Gap Conjecture}
One of the best-known open problems in geometric group theory is the \emph{Gap Conjecture} $GC(\beta)$ \cite{GriGap}. It postulates that if the growth rate of a finitely generated group $G$ is $o(e^{\sqrt{n}})$ (or more generally, $o(e^{n^\beta})$ for a fixed $\beta\leq 1/2$), then \g has polynomial growth. The analogous question for $\epdim(G)$ is 

\begin{conjecture} \label{my gap}
There is $\beta>0$ \st\ $\epdim(G)<\beta$ implies $\epdim(G)=0$ for every finitely generated group $G$.
\end{conjecture}

Moreover, the following seems plausible:
\begin{conjecture} \label{my gap}
For every finitely generated group \g of super-polynomial growth we have
$\epdim(G)> 0$.
\end{conjecture}

Note that these two conjectures combined would imply the Gap Conjecture $GC(\beta)$ by \Prr{prop UBE}.

\section{Percolative Groups} \label{sec PG}

\Cr{corol} is proved by considering $\mu_n$ to be the equidistribution in the ball of radius $n$ of a \Cg\ of \G, and \Tr{thm nona} follows a similar idea. This raises the question of whether all groups have this property. Let me state it more precisely.

Let  $X=Cay(G,S)$ be a finitely generated \Cg\ of a group \G. Let $\mu_n, n\in \N$ be the equidistribution on the ball of radius $n$ of $X$. Call $X$ \defi{percolative}, if $\lim \la_c(\mu_n)=1$. Call \g weakly percolative if it has a percolative \Cg, and call \g (strongly) percolative if each of its finitely generated \Cg s is percolative. 

\begin{question}
Is every group weakly/strongly percolative?
\end{question}

I suspect that the groups of \cite{KasPakOsc,WilUni} fail to be percolative. 

Spanos \& Tointon \cite[Theorem 1.1]{SpaToiSpr} proved that virtually nilpotent groups are strongly percolative. The following proposition provides the lower bound needed for the proof of \Cr{corol}.

\begin{proposition} \label{prop LB}
Let $X$ be a percolative \Cg\ of a group \g \st\ $|B_X(n)| > b n^s$  (respectively, $|B_X(n)| > b^{|n|^{s}}$) for some $b,s\in \R$ and infinitely many $n\in \N$. Then $\pdim(G)\geq s$ (resp.\ $\epdim(G)\geq s$).
\end{proposition}
\begin{proof}
Let $\{n_i\}_{i\in \N}$ satisfy $|B_X(n_i)| > b n_i^s$ \fe\ $i\in \N$, and let $\mu_i$ denote the equidistribution on the ball of radius $n_i$ of $X$. Thus we have $\lim \la_c(\mu_i)=1$ by our assumptions, and so it only remains to check that the $\mu_i$ have the desired decay rate. To see this, pick $g\in G$, and note that 
$$\mu_i(g) = \frac{\mathbbm{1}_{g\in B_X(n_i)}}{|B_X(n_i)|}\leq \frac1{b n_i^s} \leq b^{-1} |g|^{-s},$$
since $\mathbbm{1}_{g\in B_X(n_i)} =1$ implies $|g|\leq n_i$. Thus $\mu_i \in \cm_s^{b^{-1}}$ as desired, and hence $\pdim(G)\geq s$. 

The proof of the claim $\epdim(G)\geq s$ follows the same lines. 
\end{proof}

\section{Beyond Groups} \label{sec BG}

It is possible to modify our definitions of dimension to make them applicable to the much more general setup where the group \g is replaced by a metric space $(X,d)$ as follows. We consider sequences of  finite or infinite edge-weighted graphs $G_i= (V_i,\mu_i)$, playing the role of the $\mu$'s in \eqref{def pdim},  \eqref{def ndim}, where $V_i\subset X$ obey an arbitrary but fixed lower bound on their pairwise distances, and 
$\mu_i:V_i^2 \to \R_+$ is symmetric (we can interpret $G_i$ as an unweighted graph by letting the support of $\mu_i$ be its edge set $E_i$, but we will not use $E_i$ explicitely). As above, we let the (possibly model-dependent) dimension of $(X,d)$ be the fastest decay-rate of a family of such $(V_i,\mu_i)$ that achieve `tree-like behaviour'. Tree-likeness can still be defined using percolation, and we can define $\pdim(X,d)$ as in \Dr{def pdim}. Alternatively, if the $G_i$ are finite, then instead of an infinite component we can ask for a \defi{giant component}, i.e.\ one containing a fixed proportion of the vertices in each $V_i$,  in a sequence of models with rates $\mu_i$ where the average vertex degree converges to 1. 

Such definitions can be applied to sequences and families of metric spaces or graphs, as well as to random rooted infinite graphs. Thus one can study e.g.\  the (percolation) dimension of the family of planar graphs of maximum degree $d$, or of a unimodular random graph \cite{BenCoa}. A particularly interesting case to consider is a triangulation $X$ of $\R^2$ with uniform volume growth of order $\Theta(r^a)$ for non-integer $a>1$; such triangulations have been constructed in \cite{BeSchrRec}.


Working with finite graphs $G_i= (V_i,\mu_i)$ as above allows us to consider further models. For example, we define the \defi{spectral dimension} by interpreting tree-likeness as the requirement that our $G_i$ form a sequence of Ramanujan graphs. 

\section{Other problems} \label{probs}

We conclude with some open problems. 
\begin{problem}
Is  $\pdim(G)>0$  true \fe\ finitely generated group? 
\end{problem}

\begin{problem}
Is $\pdim(G \times H) = \pdim(G) + \pdim(H)$ for every two finitely generated groups $G,H$? (where $\times$ denotes the cartesian product). 
\end{problem}

Such statements have been studied extensively for $asdim(G)$, see  \cite{BelDraAsy}.

\medskip
A concrete group for which it would be interesting, and perhaps realistic, to determine $\pdim{}$ of is the lampligther group $L_d$ over $\mathbb{Z}^d, d\geq 1$, which is known to have dead ends of arbitrary depth \cite{CleTabDea}. We know that $\pdim(L_d)\geq d$ by \Prr{prop d}. As far as I can see, $\pdim(L_d)$ could take any value in $[d, \infty]$, and we could even have $\epdim(L_d) =1$.

\bigskip
The following is in my opinion the most interesting currently open problem about our notions: 

\begin{conjecture}
$\pdim(G), \ndim(G)$ and $\epdim(G)$ are invariant under quasi-isometries between finitely generated groups.
\end{conjecture}

\acknowledgement{I am grateful to Gabor Pete for providing the main idea of the proof of \Tr{thm nona}.}

\bibliographystyle{plain}
\bibliography{collective}
\end{document}

%% file: defs.tex
\usepackage[usenames]{color} 
\usepackage{amsthm,amssymb,amsmath,bbm,enumerate,graphicx,epsf,stmaryrd,accents}
\usepackage[bookmarks, colorlinks=false, breaklinks=true]{hyperref} 

\usepackage{authblk}

\newcommand{\comment}[1]{}
\newcommand{\COMMENT}[1]{}

\definecolor{darkgray}{rgb}{0.3,0.3,0.3}
\newcommand{\defi}[1]{{\color{darkgray}\emph{#1}}}

\newcommand{\acknowledgement}{\section*{Acknowledgement}}

\comment{
	\begin{lemma}\label{}	
\end{lemma}
\begin{proof}

\end{proof}

\begin{theorem}\label{}
\end{theorem} 
\begin{proof} 	

\end{proof}

}

\newtheorem{proposition}{Proposition}[section]
\newtheorem{definition}[proposition]{Definition}
\newtheorem{theorem}[proposition]{Theorem}
\newtheorem{corollary}[proposition]{Corollary}

\newtheorem{lemma}[proposition]{Lemma}

\newtheorem{conjecture}{{Conjecture}}[section]

\newtheorem{problem}[conjecture]{{Problem}}

\newtheorem{question}[conjecture]{{Question}}

\newtheorem{examp}[proposition]{Example}

\newcommand{\FIG}{0}

\ifnum \NOTESON = 1 \newcommand{\note}[1]{ 

\hspace*{-30pt}
	{\color{blue}  NOTE: \color{Turquoise}{\small  \tt \begin{minipage}[c]{1.1\textwidth}  #1 \end{minipage} \ignorespacesafterend }} 
	
	}
\else \newcommand{\note}[1]{} \fi

\newcommand{\afsubm}[1]{ \ifnum \Debug = 1 {\mymargin{#1}}
\fi} 

\ifnum \Debug = 1 
\else  \fi

\ifnum \FIG = 1 \newcommand{\fig}[1]{Figure ``{#1}''}
\else \newcommand{\fig}[1]{Figure~\ref{#1}} \fi

\ifnum \FIG = 1 
\else  \fi

\ifnum \Debug = 1 \usepackage[notref,notcite]{showkeys}
\fi

\ifnum \COLORON = 0 \renewcommand{\color}[1]{}
\fi

\newcommand{\N}{\ensuremath{\mathbb N}}
\newcommand{\R}{\ensuremath{\mathbb R}}

\newcommand{\Z}{\ensuremath{\mathbb Z}}

\newcommand{\cm}{\ensuremath{\mathcal M}}

\newcommand{\sm}{\backslash}

\makeatletter
\DeclareRobustCommand{\cev}[1]{%
  \mathpalette\do@cev{#1}%
}
\newcommand{\do@cev}[2]{%
  \fix@cev{#1}{+}%
  \reflectbox{$\m@th#1\vec{\reflectbox{$\fix@cev{#1}{-}\m@th#1#2\fix@cev{#1}{+}$}}$}%
  \fix@cev{#1}{-}%
}
\newcommand{\fix@cev}[2]{%
  \ifx#1\displaystyle
    \mkern#23mu
  \else
    \ifx#1\textstyle
      \mkern#23mu
    \else
      \ifx#1\scriptstyle
        \mkern#22mu
      \else
        \mkern#22mu
      \fi
    \fi
  \fi
}

\makeatother

\newcommand{\seq}[1]{\ensuremath{(#1_n)_{n\in\N}}} 

\newcommand{\g}{\ensuremath{G\ }}
\newcommand{\G}{\ensuremath{G}}

\renewcommand{\Pr}{\mathbb{P}}

\newcommand{\Cg}{Cayley graph}

\newcommand{\Tr}[1]{Theorem~\ref{#1}}

\newcommand{\Sr}[1]{Section~\ref{#1}}

\newcommand{\Prr}[1]{Pro\-position~\ref{#1}}

\newcommand{\Cr}[1]{Corollary~\ref{#1}}

\newcommand{\Dr}[1]{De\-fi\-nition~\ref{#1}}

\renewcommand{\iff}{if and only if}
\newcommand{\fe}{for every}
\newcommand{\Fe}{For every}

\newcommand{\st}{such that}

\newcommand{\ti}{there is}

\newcommand{\wrt}{with respect to}

\newcommand{\rw}{random walk}

\newcommand{\labtequ}[2]{
 \begin{equation} \label{#1} 	\begin{minipage}[c]{0.9\textwidth}  #2 \end{minipage} \ignorespacesafterend \end{equation} } 

\newcommand{\labtequc}[2]{ \begin{equation} \label{#1} 	\text{   #2 } \end{equation} }

\newcommand{\mymargin}[1]{
 \ifnum \Debug = 1
  \marginpar{%
    \begin{minipage}{\marginparwidth}\small%
      \begin{flushleft}%
        {\color{blue}#1}%
      \end{flushleft}%
   \end{minipage}%
  }%
 \fi
}%

\newcommand{\extras}[1]{
 \ifnum \Debug = 1
\section{Extras} #1
 \fi
}%

\newcommand{\mySection}[2]{}

